\documentclass{article}
\usepackage{graphicx} % Required for inserting images
\usepackage{biblatex}
\usepackage{graphicx} % Required for inserting images
\usepackage{tikz}
\usepackage{graphicx}
\usepackage{algorithm}
\usepackage{algorithmic}
\usepackage{cmap}
\usepackage{wrapfig}
\usepackage{amssymb}
\usepackage{moreverb}
\usepackage{setspace}
\usepackage{booktabs}
\usepackage{url}
\usepackage{amssymb}
\usepackage[shortlabels]{enumitem}
\usepackage{amsmath}
\usepackage{amsthm}
\usepackage{adjustbox}
\usepackage{tabularx}
\usepackage{tabularray}
\usepackage{multicol}
\usepackage{caption}
\usepackage{listings}
\usepackage{subcaption}
\usepackage{tabularx}
\usepackage[table,usenames,dvipsnames]{xcolor}
\usepackage{a4wide}
\usepackage{rotating}
\usepackage{makecell}
\usepackage{ragged2e}
\usepackage[acronym]{glossaries}
\usepackage{glossary-mcols}
\usepackage{hyperref}
\usepackage{cleveref} 
\usepackage[makeroom]{cancel}

\usepackage{a4wide, amssymb, amsthm, amsmath, amsfonts, graphicx,wrapfig,float}
\usepackage{tikz}

\newtheorem{theorem}{Theorem}[section]

\newtheorem{definition}[theorem]{Definition}
\newtheorem{lemma}{Lemma}

\newtheorem{corollary}[theorem]{Corollary}
\DeclareMathOperator{\PGL}{PGL}
\DeclareMathOperator{\SU}{SU}

\DeclareMathOperator{\tr}{tr}

\DeclareMathOperator{\GL}{GL}
\DeclareMathOperator{\SL}{SL}
\DeclareMathOperator{\PSL}{PSL}

\DeclareMathOperator\Aut{Aut}

\title{A note on the reflexibility of regular maps with  automorphism group $\PGL(2,q)$}
\author{Darius Young}
\date{June 2026}
\begin{document}
\maketitle
\begin{abstract}
    By a short adaptation of an argument of Singerman, we show that for every prime power $q$, every orientably regular map (or hypermap) whose orientation-preserving automorphism group is isomorphic to $\PGL(2,q)$ is reflexible.
\end{abstract}
\section{Introduction}

A \emph{map} $M$ is a cellular embedding of a connected graph into a closed connected surface $S$ such that $S\setminus M$ is homeomorphic to a disjoint union of disks. When the supporting surface of a map is orientable, we also call the map orientable. The orientation-preserving automorphisms of a map $M$ on an orientable surface $S$ form a group $\Aut^+(M)$ acting semiregularly on the arcs (directed edges) of $M$. If this action is transitive, and hence regular, the map is said to be \emph{orientably regular}. In this case $\Aut^+(M)$ is generated by a pair of elements $(r,s)$, where $r$ is a rotation about a face $F$, $s$ is a rotation about a vertex $V$ incident to $F$, and $rs$ is the involution fixing an edge incident to $V$ and $F$. When a regular map admits an orientation-reversing symmetry (reflection), then it is called \emph{reflexible}, and when it does not, $M$ is called \emph{chiral}. The condition that $M$ is reflexible is equivalent to the condition that $\Aut^+(M)$ admits an automorphism inverting both elements of a standard generating pair, that is, there exists some $\phi \in \Aut(\Aut^+(M))$ with $\phi(r)=r^{-1}$ and $\phi(s)=s^{-1}$.

The goal of this paper will be to extend a $1974$ theorem of Singerman \cite{SingermanSymmetries}:

\begin{theorem}[Singerman]
\label{thm:PSLReflex}
   Let $(X,Y)$ be a generating pair of $\PSL(2,q)$ for some prime power $q$. Then there exists an automorphism of $\PSL(2,q)$ mapping $X\mapsto X^{-1}$ and $Y\mapsto Y^{-1}$.
\end{theorem}
\begin{corollary}
\label{cor:RegMapReflex}
If $M$ is an orientably regular map with automorphism group isomorphic to $\PSL(2,q)$ for some prime power $q=p^n$, then $M$ must be reflexible.
\end{corollary}

\begin{proof}
    If $\PSL(2,q)$ is isomorphic to $\Aut^+(M)$ then it must be generated by a standard pair of generators $r,s$, and if there is an automorphism of $\Aut^+(M)$ which inverts both standard generators of $\Aut^+(M)$ then $M$ is reflexible.
\end{proof}

The main result of this paper extends \Cref{thm:PSLReflex} by replacing $\PSL(2,q)$ with $\PGL(2,q)$ in the statement. \Cref{thm:PSLReflex} follows readily from a result of Macbeath, and our extension is proved by a similar argument, requiring only a minor modification.

 We begin with some notation following Macbeath in \cite{Macbeath1969}. By an \textit{$\SL$-triple} we will mean an ordered triple $(X,Y,Z)$ of elements of $\SL(2,q)$ such that $XYZ=1$. By a \textit{$q$-triple}, we will mean an ordered triple of elements $(a,b,c)$ in $\mathbb{F}_q$. We associate with each $q$-triple a set $E(a,b,c)$, which is the set of all $\SL$-triples such that $(\tr \;X,\tr\;Y,\tr\;Z)=(a,b,c)$.

 The following lemma is well known and stated but not referenced or proved in \cite{Macbeath1969}. We provide a short proof here for completeness.

\begin{lemma}
\label{lem:representationSL}
Let $G_1$ be the subgroup of $\SL(2,q^2)$ generated by matrices of the form 

$$
\begin{pmatrix}
a & b\\
b^q & a^q
\end{pmatrix}
$$
where $a,b\in\mathbb{F}_{q^2}$ with $aa^q-bb^q=1$. The group $G_1$ is isomorphic to $\SL(2,q)$.
\end{lemma}

\begin{proof}
    Note that $G_1$ is isomorphic to $\SU(2,q)$, the group of determinant $1$ unitary matrices with elements in $\mathbb{F}_{q^2}$. Specificaly conjugating $G_1$ by the matrix $(1\;0\;;\;0\;\lambda)$, where $\lambda\in\mathbb{F}_{q^2}$ and $\lambda^{q+1}=-1$ we obtain the standard presentation of $\SU(2,q)$. The lemma thus follows from the fact that $\SU(2,q)$ and $\SL(2,q)$ are isomorphic \cite[pp.~66--67]{Wilson2009}.
\end{proof}

Again, following Macbeath, we define the affine subgroups of $\SL(2,q)$.

\begin{definition}
\label{def:affine}
In $\SL(2,q)$ the set of superdiagonal matrices

$$
\begin{pmatrix}
a & b\\
0 & a^{-1}
\end{pmatrix}
$$

forms a subgroup $A$. The natural map $\phi$ sends $A$ to a subgroup $A_1$ of $\PSL(2,q)$, similarly the set of matrices

$$
\begin{pmatrix}
t & 0\\
0 & t^{q}
\end{pmatrix} \hspace{1cm} t\in \mathbb{F}_{q^2},\;t^{q+1}=1
$$

forms a subgroup of $\SL(2,q^2)$ which can be mapped by the isomorphism of $G_1$ into $\SL(2,q)$ from \Cref{lem:representationSL} followed by $\phi$ to a subgroup $A_2$ of $\PSL(2,q)$. Any subgroup of a conjugate (in $\PSL(2,q)$) of
$A_1$ or $A_2$ is called an affine subgroup of $\PSL(2,q)$.
\end{definition}

Macbeath proved that for every $q$-triple, the set $E(a,b,c)$ is non-empty, \cite[Theorem 1]{Macbeath1969}. He then called a $q$-triple \textit{singular} depending on whether a particular quadratic form associated with that triple is singular, and called an $\SL$-triple \textit{singular} if its $q$-triple of traces is singular. Macbeath went on to derive the following, \cite[Theorem 2]{Macbeath1969}. 

\begin{theorem}[Macbeath]
An $\SL$-triple is singular if and only if its image in $\PSL(2,q)$ generates an \emph{affine} subgroup.
\end{theorem}

% Macbeath defined an affine subgroup of $\PSL(2,q)$ as a subgroup conjugate to a subgroup of one of two particular subgroups $A_1$ and $A_2$, \cite[p.~20, Class~II]{Macbeath1969}. Where $A_1$ is the image of the superdiagonal matrices of $SL(2,q)$ in $PSL(2,q)$ and $A_2$ is a particular cyclic group of matrices obtained via a $PSL(2,q)$ representation of $PSL(2,q^2)$.

In light of this theorem, we will take it as the definition that an $\SL$-triple is singular if it generates an affine subgroup in $\PSL(2,q)$. We will also call a $q$-triple singular if it is the triple of traces of a singular $\SL$-triple. With this terminology established, we can state Macbeath's theorem.

\begin{theorem}[Macbeath]
Let $q=p^n$ be a prime power and let $\SL(2,\bar{q})$ denote the special linear group over the algebraic closure of $\mathbb{F}_q$.

    \begin{enumerate}
        \item If $p\not = 2$ and $(a,b,c)$ is a non-singular $q$-triple, then $E(a,b,c)$ contains two $\SL(2,q)$ conjugacy classes. Any two $\SL$-triples in $E(a,b,c)$ are conjugate in $\SL(2,\bar{q})$.
        \item If $p=2$ and $(a,b,c)$ is a non-singular $q$-triple, then $E(a,b,c)$ contains only one $\SL(2,q)$ conjugacy class.
    \end{enumerate}
\end{theorem}

We now see how Singerman's theorem follows from Macbeath's. Let $\PSL(2,q)$ be generated by two elements $x$ and $y$. Taking $z=(xy)^{-1}$, we can choose elements $X,Y,Z$ in $\SL(2,q)$ as the preimages of $x,y,z$ under the canonical projection map $\pi:\SL(2,q)\rightarrow \PSL(2,q)$. The matrices $X,Y,Z$ can be chosen so that $(X,Y,Z)$ is an $\SL$-triple. This triple will clearly not generate an affine subgroup of $\SL(2,q)$, this can be seen by the fact that $\PSL(2,q)$ is not soluble for $q>3$. The exceptional cases $\PSL(2,2)$ and $\PSL(2,3)$ can be checked individually by comparing orders with $A_1$ and $A_2$ from Definition \ref{def:affine}. So the $q$-triple $(\tr X, \tr Y,\tr Z)$ is non-singular. Furthermore, since $X$ and $Y$ have determinant $1$, we know that $\tr X = \tr X^{-1}$ and $\tr Y=\tr Y^{-1}$, so both $(X,Y,Z)$, and $(X^{-1},Y^{-1},YX)$ belong to $E(\tr X, \tr Y,\tr Z)$. Hence, by Macbeath's theorem, these two $\SL$-triples are conjugate in $\SL(2,\bar{q})$, or in $\SL(2,q)$ if $q$ is even. It follows that there is an automorphism $\alpha$ of $\SL(2,\bar{q})$ mapping $X\mapsto X^{-1}$ and $Y\mapsto Y^{-1}$. Thus $\alpha$ restricts to an automorphism of $\langle X,Y\rangle$, and projects to an automorphism of $\PSL(2,q)$ mapping $x\mapsto x^{-1}$ and $y\mapsto y^{-1}$. Note that we don't actually require that $x$ and $y$ generate $\PSL(2,q)$, but only that $X,Y$ do not generate an affine subgroup of $\SL(2,q)$.

It turns out that the above argument works in almost the same way for $\PGL(2,q)$, with the only extra step being that in order to use Macbeath's theorem, we first embed $\PGL(2,q)$ into the special linear group over the quadratic extension of $\mathbb{F}_q$.

\begin{theorem}
If $(x,y)$ is a generating pair of $\PGL(2,q)$ for some prime power $q=p^n$, then there exists an automorphism of $\PGL(2,q)$ mapping $x\mapsto x^{-1}$ and $y\mapsto y^{-1}$.
\end{theorem}
\begin{proof}
Let $x,y\in \PGL(2,q)$ be a generating pair and let $z=(xy)^{-1}$, and let $Z_1$ and $Z_2$ be the respective centres of $\GL(2,q)$ and $\SL(2,q^2)$. Since $\mathbb{F}_q$ is a subfield of $\mathbb{F}_{q^2}$ we may embed $\PGL(2,q)$ into $\PSL(2,q^2)$ via the homomorphism that takes any element $AZ_1\in \PGL(2,q)$ to $\sqrt{\det(A)}^{-1}AZ_2\in \PSL(2,q^2)$, note that both the positive and negative roots map to the same element of $\PSL(2,q^2)$. Let $\hat{x},\hat{y},\hat{z}$ be the images of $x$, $y$, and $z$ under this homomorphism. Then choose elements $X,Y,Z\in\SL(2,q^2)$ from the preimages of $\hat{x},\hat{y},\hat{z}$ under the canonical projection $\pi:\SL(2,q^2)\rightarrow \PSL(2,q^2)$ so that $XYZ=1$. Note that $x$ and $y$ generate $\PGL(2,q)$, which is insoluble for $q>3$ and hence cannot be affine. Let $A_1$ and $A_2$ as in Defintion \ref{def:affine} be subgroups of $\PSL(2,q^2)$. We check that for $q\in\{2,3\}$ $\PGL(2,q)$ is still not affine, note that $A_1$ has a normal Sylow $p$-subgroup while the $\PGL(2,q)$ (isomorphic to $S_3$, $S_4$) does not, so is not contained in a conjugate of $A_1$. Since $A_2$ is cyclic, again no conjugate can contain $\PGL(2,q)$. So $(X,Y,Z)$ is a non-singular $\SL$-triple, and both $(X,Y,Z)$ and $(X^{-1},Y^{-1},YX)$ are elements of $E(\tr X,\tr Y,\tr Z)$. By Macbeath's theorem, these triples are conjugate in $\SL(2,\bar{q})$, and so there is an automorphism of $\SL(2,\bar{q})$ mapping $X\mapsto X^{-1}$ and $Y\mapsto Y^{-1}$, and then since the centre of a group is characteristic, this automorphism projects to an automorphism of $\PGL(2,q)$ mapping $x\mapsto x^{-1}$ and $y\mapsto y^{-1}$.
\end{proof}

\begin{corollary}
If $M$ is a regular map with automorphism group isomorphic to $\PGL(2,q)$ for some prime power $q$, then $M$ must be reflexible.
\end{corollary}
\begin{proof}
    Analogous to the proof of \Cref{cor:RegMapReflex}.
\end{proof}

\printbibliography

\end{document}